\documentclass[11pt]{article}
\usepackage{amscd, amsmath, amssymb, amsthm}
\usepackage[all,cmtip]{xy}
\usepackage[pagebackref]{hyperref}

\title{Rationality does not specialize among terminal varieties}
\author{Burt Totaro}
\date{  }

\def\C{\text{\bf C}}
\def\P{\text{\bf P}}

\def\arrow{\rightarrow}

\begin{document}
\maketitle
\newtheorem{theorem}{Theorem}[section]
\newtheorem{corollary}[theorem]{Corollary}
\newtheorem{lemma}[theorem]{Lemma}

\theoremstyle{definition}
\newtheorem{definition}[theorem]{Definition}
\newtheorem{example}[theorem]{Example}

\theoremstyle{remark}
\newtheorem{remark}[theorem]{Remark}

An algebraic variety is {\it rational }if it becomes isomorphic
to projective space after removing lower-dimensional
subvarieties from both sides. Little is known about
how rationality behaves in families. In particular,
given a family of projective varieties for which the geometric
generic fiber is rational, is every fiber geometrically rational?
(``Geometric'' refers to properties of a variety after
extending its base field to be algebraically closed.)

Matsusaka proved that the analogous
question for geometric ruledness has a positive answer
\cite[Theorem IV.1.6]{Kollarbook}.
(By definition, a variety is ruled if it is birational to the product
of the projective line with some variety.) That is, ruledness
specializes in families of varieties. For example,
Koll\'ar used Matsusaka's theorem
to show that a large class of Fano hypersurfaces
are not ruled and therefore not rational
\cite[Theorem V.5.14]{Kollarbook}. By contrast, rationality
does not specialize in this generality, as shown by a family
of cubic surfaces over the complex numbers $\C$
with most fibers smooth and one fiber the projective cone
over a smooth cubic curve. Every smooth cubic surface
is rational, but the cone over a smooth cubic curve $E$ is birational
to $E\times \P^1$, which is not rational because
it has a nonzero holomorphic 1-form.

Note, however, that the cone over a cubic curve has a fairly bad singularity:
it is log canonical but not klt (Kawamata log terminal). This suggests
the question of whether rationality specializes among
varieties with milder singularities. Indeed, it follows
from de Fernex and Fusi \cite[Theorem 1.3]{DFF} and Hacon
and M\textsuperscript{c}Kernan
\cite[Corollary 1.5]{HM}
that rationality specializes
among klt complex varieties of dimension at most 3.

Extending work of Voisin \cite{Voisin} and Colliot-Th\'el\`ene
and Pirutka \cite{CTP},
\cite[Theorem 2.1]{Totarohyper}
showed that a large class of Fano hypersurfaces $X$
are not stably rational. (That is,
no product of $X$ with projective space is rational.)
As an application, suggested by de Fernex,
\cite[Corollary 4.1]{Totarohyper} showed
that rationality does not specialize among klt varieties
of dimension 4 or higher.

In this paper, we find that the results of \cite{Totarohyper}
are strong enough to imply that rationality
does not specialize even among {\it terminal }varieties.
Terminal singularities
form the narrowest class of singularities that comes up
in the minimal model program.
The examples
are in any dimension at least 5. 

Some natural remaining questions are:
Does rationality specialize among terminal 4-folds?
Does rationality specialize among smooth varieties?

This work was supported by NSF grant DMS-1303105.

\section{The example}

\begin{theorem}
\label{main}
There is a flat projective morphism $f\colon  X\arrow C$
with $C$ a Zariski open subset of the complex affine line
such that $0$ is in $C$,
all fibers of $f$ have terminal singularities, all fibers of $f$
over $C-0$ are rational, and the fiber $F$ over $0$ is not rational.

Such examples exist with $F$ of any dimension at least 5.
There is also a family of 4-folds with canonical singularities over
a Zariski open subset $C$ of $A^1_{\C}$
such that all fibers over $C-0$ are rational and the fiber $F$
over $0$ is not rational.
\end{theorem}

In other words, rationality does not specialize among terminal varieties
of dimension at least 5, or among canonical varieties of dimension
at least 4. (Throughout, we are talking about families
of projective varieties.)

\begin{proof} (Theorem \ref{main})

We start with the following old observation.

\begin{lemma}
\label{mult}
If $X$ is a hypersurface of degree $d$ in $\P^{n+1}$ over a field $k$
such that $X$ has multiplicity equal to $d-1$
at some $k$-rational point $p$,
and if the singular locus
of $X$ has codimension at least 2, then $X$ is rational over $k$.
\end{lemma}

\begin{proof}
The assumption on the singular
locus ensures that $X$ is irreducible. The assumption on the multiplicity
of $X$ at $p$ implies that a general line through $p$ meets $X$
in exactly one other point. That gives a birational map over $k$
from the projective space $\P^n$ of lines through $p$ to $X$.
\end{proof}

We return to the proof of Theorem \ref{main}.
By \cite[Theorem 2.1]{Totarohyper}, a very general
quartic 4-fold in $\P^5_{\C}$
is not stably rational. Choose one smooth quartic 4-fold $Y$ over $\C$
which is not stably rational. Let $X_0$ be the projective cone
over $Y$ in $\P^6$. Then $X_0$ is a quartic 5-fold, and $X_0$
is not rational because it is birational to $\P^1\times Y$.
Also, $X_0$ is terminal, because $Y$ has Fano index 2 which is greater than 1,
meaning that the anticanonical bundle $-K_Y$ is given
by $-K_{Y}\cong -(K_{\P^5}+Y)|_Y=O(6-4)|_Y=O(2)|_Y$
\cite[Lemma 3.1]{Kollarsing}.

Let $Y$ be defined by the equation
$f_4(x_0,\ldots,x_5)=0$. Then $X_0$
is defined by the same equation in $\P^6=\{[x_0,\ldots,x_6]\}$.
Let $g_3(x_0,\ldots,x_5)$ be a nonzero cubic form over $\C$.
Consider the pencil of quartics in $\P^6$ given
by the equation
$$f_4(x_0,\ldots,x_5)+ag_3(x_0,\ldots,x_5)x_6=0$$
for $a$ in the affine line $A^1_{\C}$. This gives a flat family $f\colon
X\arrow A^1$
of hypersurfaces, and the fiber over 0 is the cone $X_0$.
Since ``terminal'' is a Zariski-open condition in families
\cite[Corollary VI.5.3]{Nakayama},
there is a Zariski open neighborhood $C$ of $0$ in $A^1$ such that
all fibers of the restricted family $f\colon X_C\arrow C$
are terminal. In particular, the fibers are normal and hence
have singular locus of codimension at least 2.

Finally, for all $a\neq 0$ in $C$, the fiber $X_a$ is a hypersurface
of degree 4 in $\P^6$ with multiplicity equal to 3 at the point
$[0,\ldots,0,1]$. By Lemma \ref{mult},
it follows that $X_a$ is rational for all $a\neq 0$
in $C$. Since $X_0$ is not rational, this completes the proof
that rationality does not
specialize among terminal varieties.

The example given is a family of 5-folds. Multiplying the family
with any projective space $\P^m$ shows that rationality does
not specialize among terminal varieties of any dimension at least 5.
(Here again, it is important that $Y$ is not stably rational,
so that $X_0\times \P^m$ is not rational.)

Finally, replace the 4-fold $Y$ by a smooth quartic 3-fold (again
called $Y$) in $\P^4_{\C}$ which is not stably rational.
Such a variety exists, by
Colliot-Th\'el\`ene and Pirutka \cite{CTP}. It follows
that the projective
cone $X_0$ over $Y$ in $\P^5$ (rather than $\P^6$) is not rational.
Since $Y$ has Fano index 1, $X_0$
has canonical
but not terminal singularities. Also, ``canonical'' is a Zariski open
condition in families \cite{Kawamata}. A pencil of hypersurfaces in $\P^5$
given by the same formula as above
shows that rationality does not specialize
among 4-folds with canonical singularities.
\end{proof}


\small \sc UCLA Mathematics Department, Box 951555,
Los Angeles, CA 90095-1555

totaro@math.ucla.edu

\begin{thebibliography}{99}

\bibitem{CTP} J.-L.~Colliot-Th\'el\`ene
and A.~Pirutka. Hypersurfaces quartiques de dimension 3:
non rationalit\'e stable. {\it Ann.\ Sci.\ \'Ecole Normale
Sup\'erieure}, to appear.

\bibitem{DFF} T.~de Fernex and D.~Fusi.
Rationality in families of threefolds. 
{\it Rend.\ Circ.\ Mat.\ Palermo }{\bf 62 }(2013), 127--135.

\bibitem{HM} C.~Hacon and J.~M\textsuperscript{c}Kernan.
On Shokurov's rational connectedness conjecture.
{\it Duke Math.\ J.\ }{\bf 138 }(2007), 119-136.

\bibitem{Kawamata} Y.~Kawamata. Deformations
of canonical singularities.
{\it J.\ Amer.\ Math.\ Soc.\ }{\bf 12 }(1999), 85--92.

\bibitem{Kollarbook} J.~Koll\'ar. {\it Rational curves
on algebraic varieties. }Springer (1996).

\bibitem{Kollarsing} J.~Koll\'ar. {\it Singularities
of the minimal model program. }Cambridge (2013).

\bibitem{Nakayama} N.~Nakayama. {\it Zariski-decomposition
and abundance. }Mathematical Society of Japan (2004).

\bibitem{Totarohyper} B.~Totaro. Hypersurfaces
that are not stably rational.
{\it J.\ Amer.\ Math.\ Soc.}, to appear.

\bibitem{Voisin} C.~Voisin. Unirational threefolds
with no universal codimension 2 cycle.
{\it Invent.\ Math.\ }{\bf 201 }(2015), 207--237.

\end{thebibliography}
\end{document}